\documentclass[11pt,reqno]{amsart}
\usepackage{graphicx}
\usepackage{float}
\usepackage[caption = false]{subfig}
\usepackage{amsmath}
\usepackage{enumerate}
\usepackage{mathtools}
\usepackage[english]{babel}
\usepackage{amsfonts,amssymb}
\usepackage{mathrsfs}
\usepackage[utf8x]{inputenc}\usepackage[english]{babel}
\usepackage[all]{xy,xypic}
\usepackage{cite}
\numberwithin{equation}{section}
\newtheorem{thm}{Theorem}[section]

\newtheorem{lem}[thm]{Lemma}

\theoremstyle{definition}

\theoremstyle{remark}

\numberwithin{equation}{section}
\DeclareMathOperator{\RE}{Re}

\begin{document}
	
	\title[\tiny{The sharp bounds of the second and third Hankel determinants for the class $\mathcal{SL}^*$}]{The sharp bounds of the second and third Hankel determinants for the class $\mathcal{SL}^*$}

		\author[S. Banga]{Shagun Banga}
	\address{Department of Applied Mathematics, Delhi Technological University, Delhi--110042, India}
	\email{shagun05banga@gmail.com}

	\author[S. Sivaprasad Kumar]{S. Sivaprasad Kumar}
	\address{Department of Applied Mathematics, Delhi Technological University, Delhi--110042, India}
	\email{spkumar@dce.ac.in}

	
\maketitle
		\begin{abstract} The aim of the present paper is to obtain the sharp bounds of the  Hankel determinants $H_2(3)$ and $H_3(1)$ for the well known class $\mathcal{SL}^*$ of starlike functions associated with the right lemniscate of Bernoulli. Further for $n=3$, we find the sharp bound of the Zalcman functional for the class $\mathcal{SL}^*$. In addition, a couple of interesting results of $\mathcal{SL}^*$ is appended at the end.
	\end{abstract}
	
	\section{Introduction}
	\label{intro}
	Let $\mathcal{A}$ be the class of analytic functions $f(z)=z+\sum_{n=2}^{\infty}a_n z^n$, defined in the open unit disk $\Delta$. The subclass $\mathcal{S}$ of $\mathcal{A}$ consists of univalent functions. We say, $f$ is subordinate to $g$, denoted by $f \prec g$, if there exists a Schwartz function $\omega$ with $\omega(0)=0$ and $|\omega(z)|<1$ such that $f(z)=g(\omega(z))$, where $f$ and $g$ are analytic functions. For each $n \geq 2$, Zalcman conjectured the following coefficient inequality for the class $\mathcal{S}$:
	\begin{align}\label{zalc}
	|a_n^2-a_{2n-1}|\leq (n-1)^2.
	\end{align}
	The above inequality also implies the Bieberbach conjecture $|a_n| \leq n$ (see \cite{bieberbach1986}). Consider the class $\mathcal{SL}^*$ \cite{sokol1996}, given by
	\begin{align*}
	\mathcal{SL}^*:=\left\{f \in \mathcal{A}:\dfrac{z f'(z)}{f(z)} \prec \sqrt{1+z}, \quad z \in \Delta\right\}.
	\end{align*}
	It is evident that if $\omega=zf'(z)/f(z)$, then the analytic characterization of the functions in $\mathcal{SL}^*$, is given by $|\omega^2-1| <1$, which in fact is the interior of the right loop of the lemniscate of Bernoulli, with the boundary equation $\gamma_1$ $: (u^2+v^2)^2-2(u^2-v^2)=0$. In 2009, Sok\'{o}\l \text{ }\cite{sokol2009} obtained the sharp bounds for $a_2$, $a_3$ and $a_4$ of functions in the class $\mathcal{SL}^*$, further it is conjectured that $|a_{n+1}|\leq 1/2n$ whenever $n \geq 1$, with the extremal function $f$ satisfying $zf'(z)/f(z)=\sqrt{1+z^n}$. Later, Shelly Verma \cite{shelly2015} gave the proof for the sharp estimate of the fifth coefficient with the extremal function for $\mathcal{SL}^*$ using the characterization of positive real part functions in terms of certain positive semi-definite Hermitian form. Sok\'{o}\l \text{ }\cite{sokol22009} also dealt the radius problems for the class $\mathcal{SL}^*$. Recently, Ali $et$ $al.$ \cite{aliravi2012} have examined the radius of starlikeness associated with the lemniscate of Bernoulli. Some differential subordination results associated with lemniscate of Bernoulli is studied in \cite{ali2012,kumar2013}.
	
	The $q^{th}$ Hankel determinant for a function $f \in \mathcal{A}$, where $q$, $n \in \mathbb{N}$ is defined as follows: 
	\begin{equation}\label{hqn}
	H_q(n) :=
	\begin{vmatrix}
	a_n & a_{n+1} & \ldots & a_{n+q-1}\\
	a_{n+1} & a_{n+2} & \ldots & a_{n+q}\\
	\vdots & \vdots & \ddots & \vdots \\
	a_{n+q-1}& a_{n+q} &\ldots & a_{n+2q-2}
	\end{vmatrix}.
	\end{equation}
	This has been initially studied in \cite{pommerenke1966}. This determinant has also been considered by several authors. It also plays an important role in the study of singularities (see \cite{diens1957}). Noor \cite{noor1983} studied the rate of growth of $H_q(n)$ as $n \rightarrow \infty$ for functions in $\mathcal{S}$  with bounded boundary. The computation of the upper bound of $|H_q(n)|$ for several subclasses of $\mathcal{S}$ has always been a trendy problem in the field of geoemteric function theory. Hayami and Owa \cite{hayami2010} determined the second Hankel determinant $H_2(n)$ $(n=1,2,\ldots)$ for functions $f$ satisfying $ \RE (f(z)/z)> \alpha$ or $ \RE f'(z) > \alpha$ $(0 \leq \alpha <1)$. Recently, Zaprawa \cite{zaprawa2016} obtained the upper bound of $|H_2(n)|$ for the class $T$ of typically real functions. Note that the Hankel determinant $H_2(1):=a_3-a_2^2$ coincides with the famous Fekete-Szeg\"o functional. In the year 1983, Bieberbach \cite{bieberbach1983} estimated the bound of $|H_2(1)|$ for the class $\mathcal{S}$. The generalization of Fekete-Szeg\"o functional is given by $a_3-\mu a_2^2$, where $\mu$ is either real or complex. The computation for the bound of $|H_2(2)|$, where $H_2(2):=a_2 a_4-a_3^2$ requires the formulae of $p_2$ and $p_3$ \cite{libera1982} 
	in terms of $p_1$, where $p_i's$ are the coefficients of the functions in the Carath\'{e}odory class $\mathcal{P}$, defined by:
	\begin{align*}
	p(z)= 1+ \sum_{n=1}^{\infty} p_n z^n \text{ } (z \in \Delta),
	\end{align*}
	with $\RE p(z) >0$ in $\Delta$. Recently, many authors have estimated the  bound of  $|H_2(2)|$ (see \cite{bansal2013,hayami2010,hayman1968,janteng2007,lee2013}). Recall that the second Hankel determinant is given by
	\begin{align}\label{h231}
	H_2(3)= 	\begin{vmatrix}
	a_3 & a_{4}\\
	a_{4} & a_{5}\\
	\end{vmatrix}
	=a_3 a_5-a_4^2.\end{align} Zaprawa \cite{zaprawa2018} investigated the Hankel determinant $H_2(3)$ for several classes of univalent functions. The estimate of the upper bound of the third order Hankel determinant, which is given by
	\begin{align}\label{h3}
	H_3(1)= 	\begin{vmatrix}
	a_1 & a_{2} & a_{3}\\
	a_{2} & a_{3} & a_{4}\\
	a_{3}& a_{4} & a_{5}
	\end{vmatrix}=a_3(a_2 a_4-a_3^2)-a_4(a_4-a_2 a_3)+a_5(a_3-a_2^2),
	\end{align} requires the sharp bounds of the initial coefficients ($a_2$, $a_3$, $a_4$ and $a_5$), Fekete-Szeg\"o functional, second Hankel determinant $H_2(2)$ and the quantity $|a_4-a_2 a_3|=:L$. Using triangle inequality in (\ref{h3}), the upper bound of $|H_3(1)|$ can be obtained as follows:
	\begin{align*}
	|H_3(1)| \leq |a_3||H_2(2)|+|a_4||L|+|a_5||H_2(1)|,
	\end{align*}
	(see \cite{krishna2015,raza2013,sudharshan2014,zaprawa2017}). Note that this computation does not yield sharp bound for $H_3(1)$. It is pertinent to know that the computation of $|H_3(1)|$ and $|H_2(3)|$ is tedious if we desire to obtain the sharp bound. For the class $\mathcal{SL}^*$, Raza and Malik \cite{raza2013} obtained the sharp bounds of $|H_2(1)|$ and $|H_2(2)|$ and the upper bound of $|H_3(1)|$. Thus, the sharp estimate of $|H_3(1)|$ for $\mathcal{SL}^*$ until now is an open problem. The study of the bound for third Hankel determinant has become an interesting problem only after the well known formula of expressing $p_4$ in terms of $p_1$ which was recently obtained in \cite{kwon2018}, which yields the sharp results in most of the cases. Kwon $et$ $al.$ \cite{kwon2019} improved the estimate of the third Hankel determinant for starlike functions. Recently, Kowalczyk $et$ $al.$ \cite{kowalczyk2018} obtained the sharp bound of $|H_3(1)|$ for the class $\mathit{T}(\alpha) :=\{f \in \mathcal{A}:\RE (f(z)/z) > \alpha \text{, }\alpha \in [0,1)\}$ and in \cite{kowalczyk22018} establish the sharp bound of the same for the class of convex functions. Zaprawa \cite{zaprawa2016} estimated the sharp bound of $|H_2(3)|$ for the class of typically real functions. Note that these are the only three (as per the knowlegde of the authors) sharp bounds of $|H_2(3)|$ and $|H_3(1)|$ proved for any subclass of analytic functions till date.
	
	For the class $\mathcal{SL}^*$, the known upper bound for $|H_3(1)|$ is $\tfrac{43}{576}$ (see \cite{raza2013}), whereas in this paper, we obtain a sharp estimate for the same which is equal to $\tfrac{1}{36}$. Further, we find the sharp bound of the second Hankel determinant $H_2(3)$ for the class $\mathcal{SL}^*$. Also, we estimate the sharp bound of the quantity $|a_3^2-a_5|$  for the class $\mathcal{SL}^*$, which is the Zalcman functional, given in (\ref{zalc}), when $n=3$. In the last section, we establish few results pertaining to the sufficient condition for the functions in $\mathcal{S}$ to belong to the class $\mathcal{SL}^*$. 
	
	The following lemmas are required for the formulae of $p_2$, $p_3$ \cite{libera1982} and $p_4$ \cite{kwon2018} in order to establish our main results.
	\begin{lem} \label{p1p2p3}
		Let $p \in \mathcal{P}$ and of the form $1+\sum\limits_{n=1}^{\infty}p_nz^n.$ Then
		\begin{align}\label{p2}
		2p_2=p_1^2+\gamma(4-p_1^2),
		\end{align}
		\begin{align}\label{p3}
		4p_3=p_1^3+2p_1(4-p_1^2)\gamma-p_1(4-p_1^2)\gamma^2+2(4-p_1^2)(1-|\gamma|^2)\eta
		\end{align}
		and
		\begin{align}\label{p4}
		8p_4=&p_1^4+(4-p_1^2)\gamma(p_1^2(\gamma^2-3\gamma+3)+4\gamma)\nonumber\\&-4(4-p_1^2)(1-|\gamma|^2)(p_1(\gamma-1)\eta+\overline{\gamma}\eta^2-(1-|\eta|^2)\rho),
		\end{align}
		for some $\rho$, $\gamma$ and $\eta$ such that $|\rho|\leq1$, $|\gamma|\leq1$ and $|\eta|\leq 1.$
	\end{lem}
	\begin{lem}\label{she}\cite{shelly2015}
		Let $a$, $b$, $c$ and $d$ satisfy the inequalities $0<c<1$, $0<d<1$ and
		\begin{align*}
		8d(1-d)((cb-2a)^2+(c(d+c)-b)^2)+c(1-c)(b-2dc)^2 \leq 4c^2(1-c)^2d(1-d).
		\end{align*}
		If $p \in \mathcal{P}$, then
		\begin{align*}
		|ap_1^4+dp_2^2+2cp_1p_3-(3/2)bp_1^2p_2-p_4| \leq 2.
		\end{align*}
	\end{lem}
	
	\section{Main Results}
	\label{sec:2}
	We proceed with the following theorem.
	\begin{thm}\label{thm1}
		If $f \in \mathcal{SL}^*.$ Then we have
		\begin{align}\label{bound}
		|H_3(1)| \leq 1/36.
		\end{align}
		The bound is sharp.
	\end{thm}
	\begin{proof}
		Let $f \in \mathcal{SL}^*$ then from \cite[p. 509]{shelly2015}, we have
		\begin{align}\label{a2}
		a_2=\dfrac{p_1}{4}, \quad a_3=\dfrac{1}{8}p_2-\dfrac{3}{64}p_1^2, \quad a_4=\dfrac{1}{12}p_3-\dfrac{7}{96} p_1 p_2+\dfrac{13}{768}p_1^3\end{align}
		and
		\begin{align}\label{a5}
		a_5=-\dfrac{1}{16}\left(\dfrac{49}{384}p_1^4-\dfrac{17}{24} p_1^2p_2+\dfrac{1}{2}p_2^2+\dfrac{11}{12}p_1p_3-p_4\right).\end{align}
		On simplifying the equation (\ref{h3}), we get
		\begin{align}\label{1}
		H_3(1)= 2a_2 a_3a_4-a_3^3-a_4^2+a_3a_5-a_2^2a_5.\end{align}
		Since the class $\mathcal{P}$ is invariant under the rotation, the value of $p_1$ lies in the interval $[0,2]$. Let $p:=p_1$ and  substituting the above values of $a_i's$ in (\ref{1}), we have
		\begin{align*}
		H_3(1)=&\dfrac{1}{2359296}\bigg(689 p^6 - 3368 p^4 p_2 + 3520 p^3 p_3 + 24064 p p_2 p_3 + 3008p^2 p_2^2\\& - 16128 p^2 p_4 - 13824 p_2^3 - 16384 p_3^2 +18432 p_2 p_4\bigg).
		\end{align*}
		Using the equalities (\ref{p2})-(\ref{p4}) and upon simplification, we arrive at
		\begin{align*}
		H_3(1)=\dfrac{1}{2359296}\bigg(\nu_1(p,\gamma)+\nu_2(p,\gamma)\eta+\nu_3(p,\gamma)\eta^2+\psi(p,\gamma,\eta)\rho\bigg).
		\end{align*}
		Where $\rho$, $\eta$, $\gamma$ $\in \overline{\Delta},$
		\begin{align*}
		&\nu_1(p,\gamma):=29p^6+ (4-p^2)((4-p^2)(944p^2\gamma^2-640p^2\gamma^3-2304\gamma^3+128p^2\gamma^4)\\& \qquad \qquad \quad-116p^4\gamma+752p^4\gamma^2 -3456p^2\gamma^2-864p^4\gamma^3),\\&\nu_2(p,\gamma):=(4-p^2)(1-|\gamma|^2)\left(224 p^3+3456p^3\gamma+(4-p^2)(2432  p\gamma-512p \gamma^2)\right),\\& \nu_3(p,\gamma):=(4-p^2)(1-|\gamma|^2)\left((4-p^2)(4096-512|\gamma|^2)+3456p^2\overline{\gamma}\right),\\
		&\psi(p,\gamma,\eta):= (4-p^2)(1-|\gamma|^2)(1-|\eta|^2)\left(-3456p^2+4608\gamma(4-p^2)\right).
		\end{align*}
		Further, by taking $x:=|\gamma|$, $y:=|\eta|$ and using the fact $|\rho|\leq 1$, we have
		\begin{align*}
		|H_3(1)| &\leq \dfrac{1}{2359296}\bigg(|\nu_1(p,\gamma)|+|\nu_2(p,\gamma)|y+|\nu_3(p,\gamma)|y^2+|\psi(p,\gamma,\eta)|\bigg)\\& \leq G(p,x,y),
		\end{align*}
		where
		\begin{align}\label{G}
		G(p,x,y):=\dfrac{1}{2359296}\bigg(g_1(p,x)+g_2(p,x)y+g_3(p,x)y^2+g_4(p,x)(1-y^2)\bigg)\end{align} with
		\begin{align*}
		&g_1(p,x):=29 p^6+(4-p^2)((4-p^2)(944p^2x^2+640 p^2 x^3+2304x^3+128p^2x^4)\\& \qquad \qquad \quad +116p^4x+752p^4x^2+3456p^2x^2+864p^4x^3),\\&g_2(p,x):=(4-p^2)(1-x^2)(224p^3+(4-p^2)(2432px+512px^2)+3456p^3x),\\&g_3(p,x):=(4-p^2)(1-x^2)((4-p^2)(4096+512x^2)+3456p^2x),\\ &g_4(p,x):=(4-p^2)(1-x^2)(3456p^2+4608x(4-p^2)).\end{align*}
		Now we need to maximize $G(p,x,y)$ in the closed cuboid $S: [0,2] \times [0,1] \times [0,1]$. We establish this by finding the maximum values in the interior of the six faces, on the twelve edges  and in the interior of  $S$.\\
		\indent I. First we proceed with interior points of $S.$ Let $(p,x,y) \in (0,2)\times(0,1)\times(0,1)$. In an attempt to find the points where the maximum value is attained in the interior of $S$, we partially differentiate equation (\ref{G}) with respect to $y$ and on algebraic simplification, we get
		\begin{align*}
		\dfrac{\partial G}{\partial y}=&\dfrac{1}{73728}(4-p^2)(1-x^2)(8y(x-1)(4(4-p^2)(x-8)+27 p^2)\\&+p(4x(4-p^2)(19+4x)+p^2(7+108x))).
		\end{align*}
		Now $\tfrac{\partial G}{\partial y}=0$ yields
		\begin{align*}
		y=\dfrac{p(4x(4-p^2)(19+4x)+p^2(7+108x))}{4(x-1)(4(4-p^2)(8-x)-27p^2)}=:y_0.
		\end{align*}
		For the existence of the critical points, $y_0$ should lie in the interval $(0,1)$, which is possible only when
		\begin{align}\label{cond}
		p^3(7+108x)+4px(4-p^2)(19+4x)+32(1-x)(8-x)(4-p^2) < 216 p^2(1-x)
		\end{align}
		and
		\begin{align}\label{cond2}
		27p^2 > 4(4-p^2)(8-x). 
		\end{align}
		Now, we find the solutions satisfying both the inequalities (\ref{cond}) and (\ref{cond2}) for the existence of critical points. Let $g(x):= 16(8-x)/(59-4x),$ which is decreasing function of $x$ as $g'(x)$ is negative for $x \in (0,1)$. Hence $\min g(x)_{(x=1)} =112/55$. Thus from equation (\ref{cond2}), we can conclude that $p >1$ for all $x \in (0,1)$. But for $p\geq 1$, the inequality (\ref{cond}) does not hold as it is not difficult to see that $7 p^3\geq 216p^2(1-x)$ for all $x$. This shows that there does not exist any solution satisfying both the inequalities (\ref{cond}) and (\ref{cond2}). Hence the function $G$ has no critical point in $(0,2) \times (0,1) \times (0,1)$. \\
		\indent II. Here we consider the interior of all the six faces of the cuboid $S$.\\
		
		On the face $p=0,$ $G(p,x,y)$ reduces to
		\begin{align}\label{f4}
		h_1(x,y):= G(0,x,y)=\dfrac{2(1-x^2)(y^2(x-1)(x-8)+9x)+9x^3}{576}, \quad x, y \in (0,1).
		\end{align}
		We note that $h_1$ has no critical point in $(0,1)\times (0,1)$ since
		\begin{align}\label{part}
		\dfrac{\partial h_1}{\partial y} =\dfrac{ y(1-x^2)(x-1)(x-8)}{144} \neq 0, \quad x, y \in (0,1).
		\end{align}
		\indent On the face $p=2,$ $G(p,x,y)$ reduces to
		\begin{align}\label{f3}
		G(2,x,y)= \dfrac{29}{36864}, \quad x,y \in (0,1).
		\end{align}
		\indent On the face $x=0$, $G(p,x,y)$ reduces to $G(p,0,y)$, given by
		\begin{align}\label{f1}
	h_2(p,y) :=	&\dfrac{128 y^2(512 - 364 p^2 + 59 p^4)+224p^3y(4-p^2)+13824 p^2 - 3456 p^4+29p^6}{2359296},
		\end{align}
		where $p \in (0,2) \text{ and } y \in (0,1)$. We solve $\tfrac{\partial h_2}{\partial y}=0$ and $\tfrac{\partial h_2}{\partial p}=0$ to determine the points where the maxima occur. On solving $\tfrac{\partial h_2}{\partial y}=0$, we get
		\begin{align}\label{p}
		y=-\dfrac{7 p^3}{8(128-59 p^2)}=:y_1.
		\end{align}
		For the given range of $y$, we should have $y_1 \in (0,1)$, which is possible only if $ p > p_0$, $p_0 \approx 1.47292$.
		A computation shows that $\tfrac{\partial h_2}{\partial p}=0$ implies
		\begin{align}\label{f}
		256 y^2(-182 + 59 p^2)-112 y(-12 p + 5 p^3)+87 p^4-6912p^2+13824=0.
		\end{align}
		Substituting equation (\ref{p}) in equation (\ref{f}) and upon simplification, we get
		\begin{align}\label{psol}
		75497472-107347968p^2+51265024p^4-8426096p^6+95167p^8=0.
		\end{align}
		A numerical computation shows that the solution of (\ref{psol}) in the interval $(0,2)$ is $p \approx 1.39732 $. Thus $h_2$ has no critical point in $(0,2) \times (0,1)$.\\
		\indent On the face $x=1,$ $G(p,x,y)$ reduces to
		\begin{align}\label{f2}
		h_3(p,y):= G(p,1,y)= \dfrac{36864 + 22784 p^2 - 7920 p^4 + 9 p^6}{2359296}, \quad p \in (0,2).\end{align}
		Solving $\tfrac{\partial h_3}{\partial p}=0$, we get a critical point at $p=:p_0 \approx 1.2008$. A Simple calculation shows that $h_3$ attains its maximum value $\approx 0.0225817$ at $p_{0}$.\\
		\indent On the face $y=0$, $G(p,x,y)$ reduces to
		\begin{align*}
	h_4(p,x) :=G(p,x,0)&= \dfrac{1}{2359296}\bigg(29 p^6+(4-p^2)((4-p^2)(944 p^2 x^2+640 p^2 x^3-2304x^3\\&\quad+128 p^2 x^4+4608 x)+116 p^4 x+752 p^4 x^2 +864 p^4 x^3+3456 p^2x^2\bigg).
		\end{align*}
		A computation shows that
		\begin{align*}
		\dfrac{\partial h_4}{\partial x}=&\dfrac{1}{2359296}\bigg((8192p^2 - 576p^4  + 512p^6) x^3+(30720 p^2  - 4992 p^4 - 672 p^6) x^2\\&+(30208 p^2  -9088 p^4 + 384p^6) x+73728 - 36864p^2 + 5072 p^4 - 116 p^6\bigg).
		\end{align*}
		and
		\begin{align*}
		\dfrac{\partial h_4}{\partial p}=&\dfrac{1}{2359296}\bigg((4096 p - 4096p^3 +768p^5)x^4+(3840 p -6656 p^3  - 1344 p^5) x^3\\&+(30208 p- 18176 p^3+ 1152p^5) x^2+(-73728p + 20288 p^3  -696 p^5) x\\&+1344p - 13824 p^3+ 174 p^5\bigg).
		\end{align*}
		A numerical computation shows that there does not exist any solution for the system of equations $\tfrac{\partial h_4}{\partial x}=0$ and $\tfrac{\partial h_4}{\partial p}=0$ in $(0,2) \times (0,1)$.\\
		\indent On the face $y=1$, $G(p,x,y)$ reduces to
		\begin{align*}
		G(p,x,1)&= \dfrac{1}{2359296}\bigg(29 p^6+(4-p^2)(116 p^4 x+752 p^4 x^2+3456p^2x^2+864 p^4x^3\\&\quad+(1-x^2)(224p^3+3456p^2x+3456p^3x)+(4-p^2)((1-x^2)(2432px\\&\quad+512px^2+4096+512x^2)+944 p^2 x^2+640 p^2 x^3 +2304 x^3+128 p^2 x^4))\bigg)\\&=:h_6(p,x).
		\end{align*}
		Proceeding on the similar lines as in the previous case for face $y=0$, again there is no solution for the system of equations $\tfrac{\partial h_6}{\partial x}=0$ and $\tfrac{\partial h_6}{\partial p}=0$ in $(0,2) \times (0,1)$.\\
		\indent III. Now we calculate the maximum values achieved by $G(p,x,y)$ on the edges of the cuboid $S$.\\
		Considering the equation (\ref{f1}), we have $G(p,0,0)=:s_1(p)=(29p^6-3456p^4+13824p^2)/2359296$. It is easy to verify that the function $s_1'(p)=0$ for $p=:\lambda_0=0$ and $p=:\lambda_1 \approx 1.43285$ in the interval $[0,2]$. We observe that $\lambda_0$ is the point of minima and the maximum value of $s_1(p)$ is $\approx 0.00596162$, which  is attained at $\lambda_1$. Hence
		\begin{align*}
		G(p,0,0) \leq 0.00596162, \quad p \in [0,2].
		\end{align*}
		Evaluating the equation (\ref{f1}) at $y=1$, we obtain $G(p,0,1)= s_2(p):= (65536 -32768 p^2 + 896 p^3 +4096 p^4 - 224 p^5+ 29 p^6)/2359296$. It is easy to verify that $s_2'(p)$ is decreasing  function in $[0,2]$ and hence attains its maximum value at $p=0$.Thus
		\begin{align*}
		G(p,0,1) \leq \dfrac{1}{36}, \quad p \in [0,2].
		\end{align*}
		In view of the equation (\ref{f1}) and by straightforward computation the maximum value of $G(0,0,y)$ is attained at $y=1$. This implies
		\begin{align*}
		G(0,0,y) \leq \dfrac{1}{36}, \quad y \in [0,1].
		\end{align*}
		As the equation (\ref{f2}) is independent of $x$, we have $G(p,1,1)=G(p,1,0)=s_3(p):=(9p^6-7920 p^4+22784 p^2+36864)/2359296$. Now, $s_3'(p)=45568 p-31680 p^3 + 54 p^5=0$ for $p=:\lambda_2=0$ and $p=:\lambda_3 \approx 1.2008$ in the interval $[0,2]$, where $\lambda_2$ is a point of minima and $s_3(p)$ attains its maximum value at $\lambda_3$. We can conclude that
		\begin{align*}
		G(p,1,1)=G(p,1,0)\leq 0.0225817, \quad p \in [0,2].
		\end{align*}Substituting $p=0$ in equation (\ref{f2}), we obtain $G(0,1,y)=1/64$. The Equation (\ref{f3}) is independent of all the variables $p$, $x$ and $y$. Thus the value of $G(p,x,y)$ on the edges $p=2$, $x=1$; $p=2$, $x=0$; $p=2$, $y=0$ and $p=2$, $y=1$, respectively, is given by
		$$G(2,1,y)=G(2,0,y)=G(2,x,0)=G(2,x,1)=29/36864, x,y \in [0,1].$$ Equation (\ref{f1}), yields $G(0,0,y)= y^2/36.$ A simple computation shows that $$ G(0,0,y) \leq \dfrac{1}{36}, \quad y \in [0,1].$$
		Using equation (\ref{f4}), we get $G(0,x,1)=:s_4(x)=(16-4x^2+9x^3-2x^4)/576$. A simple computation shows that the function $s_4$ is decreasing in $[0,1]$ and hence attains its maximum value at $x=0$. Thus $$G(0,x,1) \leq \dfrac{1}{36}, \quad x \in [0,1].$$
		Once again, by using the equation (\ref{f4}), we get $G(0,x,0)=s_5'(x):=-(x^2-2)/64.$ Performing a simple calculation, we get $s_5'(x)=0$ for $x=:x_0=\sqrt{2}/\sqrt{3}$ and for $0 \leq x<x_0,$ $s_5$ is an increasing function and for $x_0<x\leq 1$, it's a decreasing function. Thus, it attains maximum value at $x_0$. Hence
		\begin{align*}
		G(0,x,0) \leq 0.0170103, \quad x \in [0,1].
		\end{align*}
		\indent In view of the cases I-III, the inequality (\ref{bound}) holds. Let the function $f : \Delta \rightarrow \mathbb{C}$ be as follows
		\begin{align}\label{fn}
		f(z)=z \exp\left(\int_{0}^{z}\dfrac{\sqrt{1+t^3}-1}{t}dt
		\right)=z+\dfrac{z^4}{6}+\cdots.\end{align}The sharpness of the bound $|H_3(1)|$ is justified by the extremal function $f$ given by (\ref{fn}), which belongs to the class $\mathcal{SL}^*$. For this function $f$, we have $a_2=a_3=a_5=0$ and $a_4=1/6$, which clearly shows that $|H_3(1)|= 1/36$ using equation (\ref{1}). This completes the proof.\qedhere
	\end{proof}
	
	We now estimate the bound for the Hankel determinant $H_2(3).$
	\begin{thm}\label{thm2}
		Let $f \in \mathcal{SL}^*$. Then we have
		\begin{align}\label{h23}
		|H_2(3)| \leq \dfrac{1}{36}.
		\end{align}
		The result is sharp.
	\end{thm}
	\begin{proof}
		We proceed here on the similar lines as in the proof of Theorem~\ref{thm1}. Now, substituting the equalities (\ref{a2})-(\ref{a5}) in (\ref{h231}) and with the assumption $p_1 =: p \in [0,2]$, we get
		\begin{align}\label{exp}
		H_2(3)= \dfrac{1}{1179648}\bigg(&103 p^6 -712 p^4 p_2 -4608 p_2^3+1984 p^2 p_2^2 +5888 p p_2 p_3\nonumber\\&-160 p^3 p_3 -8192 p_3^2 -3456 p^2 p_4 +9216 p_2 p_4\bigg).	\end{align}
		Using the equalities (\ref{p2})-(\ref{p4}) and simplifying the terms in the expression (\ref{exp}), we get
		\begin{align*}
		H_2(3)= \dfrac{1}{1179648}\bigg(\zeta_1(p,\gamma)+\zeta_2(p,\gamma)\eta +\zeta_3(p,\gamma) \eta^2 + \xi (p,\gamma,\eta)\rho\bigg),
		\end{align*}
		where $\rho$, $\eta$ and $\gamma \in \overline{\Delta},$
		\begin{align*}
		&\zeta_1(p,\gamma):= -5p^6+4p^2\gamma(4-p^2)(-p^2-20(4-p^2)\gamma-26p^2\gamma+144\gamma+36p^2\gamma^2\\&\qquad \qquad \quad +16\gamma^2(4-p^2)+40\gamma^2(4-p^2)),\\&
		\zeta_2(p,\gamma):=16p(4-p^2)(1-|\gamma|^2)(-5 p^2-36 p^2\gamma-16\gamma^2(4-p^2)-20\gamma(4-p^2)),\\&
		\zeta_3(p,\gamma):= 64(4-p^2)(1-|\gamma|^2)(-4(4-p^2)(8+\gamma^2)-9p^2\gamma),\\&
		\xi(p,\gamma,\eta):=576(4-p^2)(1-|\gamma|^2)(1-|\eta|^2)(p^2+4\gamma(4-p^2)).
		\end{align*}
		By taking $x:=|\gamma|$, $y:=|\eta|$ and using the fact $|\rho| \leq 1$, we get
		\begin{align*}
		|H_2(3)|& \leq\dfrac{1}{1179648} \bigg(|\zeta_1(p,\gamma)|+|\zeta_2(p,\gamma)|y+|\zeta_3(p,\gamma)|y^2+|\xi(p,\gamma,\eta)|\bigg)\\&\leq F(p,x,y),
		\end{align*}
		where
		\begin{align}\label{F}
		F(p,x,y):=\dfrac{1}{1179648}\bigg(q_1(p,x)+q_2(p,x)y+q_3(p,x)y^2+q_4(p,x)(1-y^2)\bigg)
		\end{align}
		with
		\begin{align*}
		&q_1(p,x):= 5p^6+4p^2x(4-p^2)(p^2+20 (4-p^2)x+26p^2x+144x+36p^2x^2\\&\qquad \qquad \quad +16x^3(4-p^2)+40x^2(4-p^2)),\\ & q_2(p,x):= 16p(4-p^2)(1-x^2)(5p^2+36 p^2x +16 x^2(4-p^2)+20 x(4-p^2)),\\& q_3(p,x):= 64(4-p^2)(1-x^2)(4(4-p^2)(8+x^2)+9p^2x),\\&q_4(p,x):= 576 (4-p^2)(1-x^2)(p^2+4x(4-p^2)).
		\end{align*}
		In order to complete the proof, we need to maximize the function  $F(p,x,y)$ in the closed cuboid $T: [0,2] \times [0,1] \times [0,1]$. For this, we find the maximum values of $F$ in $T$ by considering all the twelve edges, interior of the six faces and in the interior of $T$.\\
		\indent I. We proceed with interior points of $T$. Let us assume $(p,x,y) \in (0,2)\times(0,1)\times(0,1).$ To determine the points where the maximum value occur in the interior of $T$, we partially differentiate equation (\ref{F}) with respect to $y$ and we get
		\begin{align*}
		\dfrac{\partial F}{\partial y} =&\dfrac{1}{73728} (4-p^2)(1-x^2)(8y(x-1)(4(4-p^2)(x-8)+9p^2)\\&+p(4x(4-p^2)(5+4x)+p^2(5+36x))).
		\end{align*}
		Now, $\tfrac{\partial F}{\partial y}=0$ yields
		\begin{align*}
		y= \dfrac{p(4x(4-p^2)(5+4x)+p^2(5+36x))}{8(x-1)(4(4-p^2)(8-x)-9p^2} =: y_1.
		\end{align*}
		Now, $y_1$ should lie in the interval $(0,1)$ for the existence of the critical points. Thus, we have
		\begin{align}\label{c1}
		p^3(5+36x)+4px(4-p^2)(5+4x)+32(1-x)(8-x)(4-p^2) < 72p^2(1-x)
		\end{align}
		and
		\begin{align}\label{c2}
		4(4-p^2)(8-x)< 9 p^2.\end{align}
		We try to find the solutions satisfying  both the inequalities (\ref{c1}) and (\ref{c2}). Let us assume $g(x):= 16(8-x)/(41-4x)$, which is decreasing function of $x$ due to the fact that $g'(x)$ is negative for $ x \in (0,1)$. Therefore $\min r(x)_{(x=1)} = 112/37$. This implies $ p >1$ for all  $x \in (0,1)$ using equation (\ref{c2}). But for $p\geq 1$, the inequality (\ref{c1}) does not hold as $5 p^3\geq 72p^2(1-x)$ for all $x$. Thus we can conclude that there does not exist any solution satisfying  (\ref{c1}) and (\ref{c2}). Thus function $F$ has no critical point in $(0,2)\times(0,1)\times(0,1)$.\\
		\indent II. Now, we consider the interior of all the six faces of the cuboid $T$.\\
		On the face $p=0$,
		\begin{align}\label{g4}
		k_1(x,y):=F(0,x,y)=\dfrac{1-x^2}{288}\bigg(y^2(x-1)(x-8)+9x\bigg), \quad x,y \in (0,1).
		\end{align}
		A simple calculation shows that $\partial k_1/\partial y = \partial h_1/\partial y$. Thus equation (\ref{part}) implies $k_1$ has no critical point in $(0,1)\times (0,1)$.\\
		\indent On the face $p=2$,
		\begin{align}\label{g3}
		F(2,x,y)= \dfrac{5}{18432}, \quad x,y \in (0,1).
		\end{align}
		\indent On the face $x=0$,
		\begin{align}\label{y}
		k_2(p,y):=F(p,0,y)&=\dfrac{64 y^2(512 - 292 p^2 + 41 p^4)+80p^3y(4-p^2)+2304 p^2 -576 p^4+5p^6}{1179648},
		\end{align}
		$p \in (0,2)$ and $y \in (0,1)$. On solving $\tfrac{\partial k_2}{\partial y}=0$, we get
		\begin{align}\label{y1}
		y=\dfrac{5p^3}{8(41p^2-128)}=:y_1.
		\end{align}
		For the given range of $y$, $y_1$ should lie in the interval $(0,1)$, which holds only if $ p > p_0$, $p_0 \approx 1.7669$.
		The computation shows that $\tfrac{\partial k_3}{\partial p}=0$ implies
		\begin{align}\label{k3}
		y^2(5248p^2-18688)+40 y(12 p - 50 p^3)+2304-1152p^2+15 p^4 =0.
		\end{align}
		Let $p>p_0$ and substituting equation (\ref{y1}) in equation (\ref{k3}) and performing lengthy computation, we get
		\begin{align}\label{psol1}
		1048576-1196032 p^2+449216 p^4-57582p^6+615 p^8=0.
		\end{align}
		The numerical computation shows that the solution of (\ref{psol1}) for $p \in (0,2)$ is $p=:p_0 \approx 1.35957$. Thus $k_2$ has no critical point in $(0,2) \times (0,1)$.\\
		\indent On the face $x=1$,
		\begin{align}\label{g1}
		k_3(p):=F(p,1,y)=\dfrac{7168 p^2-2000p^4+57p^6}{1179648}, \quad p \in (0,2).
		\end{align}
		To attain maximum value of $k_3$, we solve $\partial k_3/\partial p=0$ and get critical point at $p=:p_0 \approx 1.39838$. Simple calculation shows that $k_3$ attains its maximum value $\approx 0.00576045$ at $p_0$.\\
		\indent On the face $y=0$,
		\begin{align*}
		F(p,x,0)= &\dfrac{1}{1179648}\bigg(5 p^6	+(4-p^2)((4-p^2)(2304x(1-x^2)+ 80 p^2 x^2\\&+ 160 p^2 x^3 +64 p^2 x^4)+4 p^4 x + 576 p^2 x^2 + 104 p^4 x^2 \\&+ 144 p^4 x^3 +576 p^2 (1-x^2))=:k_4(p,x).
		\end{align*}
		A complex computation shows that
		\begin{align*}
		\dfrac{\partial k_4}{\partial p}= &\dfrac{1}{589824}\bigg(2304p-1152p^3+15p^5+(-18432p+4640 p^3-12 p^5)x\\&+(1280 p-448 p^3-72 p^5)x^2+(20992p-6016 p^3+48 p^5)x^3\\&+(1024 p-1024 p^3+192 p^5)x^4\bigg)
		\end{align*}
		and
		\begin{align*}
		\dfrac{\partial k_4}{\partial x}= &\dfrac{1}{294912}\bigg((p^2-4)((-256 p^2+64 p^4)x^3+(6912-2208 p^2+12 p^4)x^2\\&+(-160 p^2-12 p^4)x-2304+576 p^2-p^4)\bigg).
		\end{align*}
		The numerical computation shows that there does not exist any solution for the system of equations $\tfrac{\partial k_5}{\partial p}=0$ and $\tfrac{\partial k_5}{\partial x}=0$ in $(0,2)\times (0,1)$.\\
		\indent On the face $y=1$,
		\begin{align*}
		F(p,x,1)&=\dfrac{1}{1179648}\bigg(5p^6+(4-p^2)((4-p^2)(80p^2x^2+64p^2x^4+160p^2x^3\\&\quad+(1-x^2)(256px^2+320 px+256(8+x^2)))+4p^4x+104p^4x^2\\&\quad+576p^2x^2+144p^4x^3+(1-x^2)(80p^3+576p^3x+576p^2x))\bigg)=:k_5(p,x).\end{align*}
		Proceeding on the similar lines as in the previous case on the face $y=0$, again, the system of equations $\partial k_5/\partial p=0$ and $\partial k_5/\partial x=0$ have no solution in $(0,2)\times (0,1)$.\\
		\indent III. We now consider the maximum values attained by $F(p,x,y)$ on the edges of the cuboid $T$:\\
		In view of the equation (\ref{y}), we have $F(p,0,0) = l_1(p):=5p^6-576p^4+2304p^2)/1179648.$ It is easy to compute that $l_1'(p)=0$ for $p=:\lambda_0=0$ and $p=:\lambda_1 \approx 1.43351$ in the interval $[0,2]$, where $\lambda_0$ is the point of minima and $\lambda_1$ is the point of maxima. Hence
		\begin{align*}
		F(p,0,0) \leq 0.00198843, \quad p \in [0,2].
		\end{align*}
		Again, considering the equation (\ref{y}), we obtain $F(p,0,1)= l_2(p):= (32768-16384p^2+320p^3+2048p^4-80p^5+5p^6)/1179648.$ Now, we note that $l_2$ is decreasing function in $[0,2]$ and hence attains its maximum value at $p=0$.Thus, 
		\begin{align*}
		F(p,0,0) \leq \dfrac{1}{36}, \quad p \in [0,2].
		\end{align*}
		Now, we observe that the equation (\ref{g1}) does not depend on the value of y, hence we get $F(p,1,1)=F(p,1,0)=l_3(p):= (7168p^2-2000p^4+57p^6)/1179648.$ It is easy to verify that the function $l_3$ has two critical points at $p=0$ and $p=:\lambda_2 \approx 1.39838$ in the interval $[0,2]$, where the maximum value is attained at $\lambda_2$. Thus
		\begin{align*}
		F(p,0,0)=F(p,1,0) \leq 0.0057645, \quad p \in [0,2].
		\end{align*}
		On substituting $p=0$ in (\ref{g1}), we get $F(0,1,y)=0$.
		In view of equation (\ref{g3}), which is independent of all the variables $p$, $x$ and $y$, the value of $F(p,x,y)$ on the edges $p=2$, $x=0$; $p=2$, $x=1$; $p=2$, $y=0$ and $p=2$, $y=1$, respectively, is given by $$F(2,0,y)=F(2,1,y)=F(2,x,0)=F(2,x,1)=5/18432, \quad x,y \in [0,1].$$
		Evaluating equation (\ref{y}) at $p=0$, we get $F(0,0,y)=l_4(y):=y^2/36$. It is easy to verify that $l_4$ is an increasing function of $y$ and hence attains maximum value at $y=1$ in $[0,1]$. Thus
		\begin{align*}
		F(0,0,y) \leq \dfrac{1}{36}, \quad y \in [0,1].
		\end{align*}
		Using equation (\ref{g4}), we get $F(0,x,1)= l_5(x):=(8-7x^2-x^4)/288$.
		Since $l_5$ is decreasing function in $[0,1]$, it attains maximum value at $x=0$. Thus
		\begin{align*}
		F(0,x,1)\leq \dfrac{1}{36}, \quad x \in [0,1].
		\end{align*}
		Substituting $y=0$ in equation (\ref{g4}), we obtain  $F(0,x,0)= l_6(x):=x(1-x^2)/32.$ A
		simple calculation shows that the function $l_6'(x)=0$ at $x=:x_0=\sqrt{3}/3$ and it is increasing in $(0,x_0)$ and decreasing in $(x_0,1)$. Hence it attains the maximum value at $x=x_0$. Thus we conclude 
		$$ F(0,x,0) \leq \sqrt{3}/144, \quad x \in [0,1].$$
		
		Taking into account all the cases I-III, the inequality $(\ref{h23})$ holds. For the function given in (\ref{fn}), which belongs to the class $\mathcal{SL}^*$, $a_3=a_5=0$ and $a_4=1/6$. Thus $|H_2(3)|=1/36$ for this function, which also proves the result is sharp. This completes the proof.\qedhere
	\end{proof}
	We note that for $n=2$, the exprssion on the left of the inequality (\ref{zalc}) reduces to the famous Fekete-Szeg\"o functional. In the following theorem we obtain the Zalcman coefficient inequality for $n=3$ for the class $\mathcal{SL}^*$.
	\begin{thm}
		Let $ f \in \mathcal{SL}^*.$ Then
		\begin{align*}
		|a_3^2-a_5| \leq \dfrac{1}{8}.
		\end{align*}
		The estimate is sharp. 
	\end{thm}
	\begin{proof}
		Using equation (\ref{a2}) and (\ref{a5}), we get
		\begin{align}\label{zal}
		a_3^2-a_5 = \dfrac{125}{12288} p_1^4-\dfrac{43}{768} p_1^2 p_2
		+\dfrac{3}{64}p_2^2+\dfrac{11}{192}p_1 p_3-\dfrac{1}{16} p_4.\end{align}
		Applying Lemma~\ref{she} with $a=125/768$, $b=43/72$, $c=11/24$ and $d=3/4$ in the equation (\ref{zal}), we get
		\begin{align*}
		|a_3^2-a_5|\leq \dfrac{1}{8}.
		\end{align*}
		Let the function $f:\Delta \rightarrow \mathbb{C}$, be defined as follows:
		\begin{align}\label{ext}
		f(z)=z \exp\left(\int_{0}^{z}\dfrac{\sqrt{1+t^4}-1}{t}dt
		\right)=z+\dfrac{z^5}{8}+\cdots.
		\end{align}
		The equality holds for the function given in (\ref{ext}), which belong to $\mathcal{SL}^*$ as $a_3=0$ and $a_5=1/8$, which contributes to the sharpness of the inequality. This completes the proof.\qedhere
	\end{proof}
	\section{Further Results}
	Let $f$ and $g$ be analytic functions of the form, respectively
	\begin{align*}
	f(z)=z + \sum_{n=2}^{\infty} a_n z^n \quad \text{ and } \quad g(z)=z+\sum_{n=2}^{\infty} b_n z^n.
	\end{align*}
	Then the Hadamard product (or convolution) of $f(z)$ and $g(z)$ is defined by
	\begin{align*}
	(f*g)(z) = z+\sum_{n=2}^{\infty} a_n b_n z^n.
	\end{align*}
	Now, we derrive the necessary and sufficient condition for a function $f \in \mathcal{S}$ to belong to the class $\mathcal{SL}^*$ in the following theorem, involving the convolution concept.
	\begin{thm}
		A function $f \in \mathcal{S}$ is in the class $\mathcal{SL}^*$ if and only if
		\begin{equation}\label{conv}
		\dfrac{1}{z}\left(f * H_t(z)\right)\neq 0, \quad (z \in \Delta)
		\end{equation}
		where $$H_t(z)= \dfrac{z}{(1-z)(1-S(t))}\left(\dfrac{1}{1-z}-S(t)\right)$$ and $$S(t)= \sqrt{t} + i\left(\pm \sqrt{\sqrt{1+4t}-(t+1)}\right),  \quad  (0 <t<2).$$
	\end{thm}
	\begin{proof}
		Define $p(z)=z f'(z)/f(z)$. As we know $p(0)=1$, to prove the result, it suffices to show that $f \in \mathcal{SL}^*$ if and only if $p(z) \notin \gamma_1$, where 
		$$ \gamma_1 =\{ (u^2+v^2)^2-2(u^2-v^2)=0\}.$$
		By taking $u^2=t$, we can give the parametric representation of the curve $\gamma_1$ as follows
		$$ S(t)= \sqrt{t} + i
		\left(\pm \sqrt{\sqrt{1+4t}-(t+1)}\right), \quad (0 < t<2).$$
		For $f \in \mathcal{S}$, we have
		\begin{align}\label{convo}
		\dfrac{z}{(1-z)^2}*f(z)=z f'(z) \quad \text{ and } \quad \dfrac{z}{1-z}*f(z)=f(z).
		\end{align}
		Using the above equations (\ref{conv}) and (\ref{convo}), we get $$ \dfrac{1}{z}\left(f * H_t(z)\right)= \dfrac{f(z)}{z (1-S(t))}\left(\dfrac{z f'(z)}{f(z)}- S(t)\right) \neq 0,$$ which clearly shows that $z f'(z)/f(z) \neq S(t)$. Hence $1/(z (f * H_t(z))) \neq 0$ if and only if $p(z) \notin \gamma_1$ if and only if $f \in \mathcal{SL}^*$.\qedhere
	\end{proof}
	\begin{thm}
		The function
		\begin{align*}
		\Theta (z)= \dfrac{z}{1-\alpha z}, \quad (z \in \Delta)
		\end{align*}
		belongs to the class $\mathcal{SL}^*$ if $|\alpha| \leq 1/4$.
	\end{thm}
	\begin{proof}
		By the definition of the class $\mathcal{SL}^*$, it suffices to show that the following inequality holds for the given range of $\alpha$.
		\begin{align}\label{alpha}
		\left|\left(\dfrac{1}{1-\alpha z}\right)^2-1\right| <1.
		\end{align}
		The above inequality (\ref{alpha}) holds\\ whenever
		\begin{align*}
		|2 \alpha z - \alpha^2 z^2| < 1+ |\alpha z|^2 -2 Re(\alpha z),
		\end{align*}
		which in turn holds if \begin{align*}2|\alpha z| \leq 1- 2 |\alpha z|,
		\end{align*} which holds if \begin{align*}
		|\alpha| \leq \dfrac{1}{4}.
		\end{align*}
		Hence the function $\Theta(z) \in \mathcal{SL}^*$.\qedhere
	\end{proof}
\noindent \textbf{Acknowledgements.} The work presented here was supported by a Research Fellowship
		from the Department of Science and Technology, New Delhi.
	
\end{document}